\documentclass[dvipdfmx]{amsart}
\usepackage{amsfonts}
\usepackage{amssymb}
\usepackage{amsmath}
\usepackage{amsthm}
\usepackage{amscd}
\usepackage{mathrsfs}
\usepackage{graphicx, color}
\usepackage{url}

\theoremstyle{theorem}
\newtheorem{theorem}{Theorem}[section]
\newtheorem{lemma}[theorem]{Lemma}

\theoremstyle{definition}

\newtheorem{remark}[theorem]{Remark}

\newcommand{\vect}[1]{
 {\boldsymbol #1}
}
\newcommand{\SmallQuandle}[2]{
 {{\sf SmallQuandle}(#1, #2)}
}

\begin{document}

\title[The knot quandle of the twist-spun trefoil is a central extension]{The knot quandle of the twist-spun trefoil is \\ a central extension of a Schl\"{a}fli quandle}

\author{Ayumu Inoue}
\address{Department of Mathematics, Tsuda University, 2-1-1 Tsuda-machi, Kodaira-shi, Tokyo 187-8577, Japan}
\email{ayminoue@tsuda.ac.jp}


\begin{abstract}
A quandle is an algebraic system which excels at describing limited symmetries of a space.
We introduce the concept of Schl\"{a}fli quandles which are defined relating to chosen rotational symmetries of regular tessellations.
On the other hand, quandles have a good chemistry with knot theory.
Associated with a knot we have its knot quandle.
We show that the knot quandle of the $m$-twist-spun trefoil is a central extension of the Schl\"{a}fli quandle related to the regular tessellation $\{ 3, m \}$ in the sense of the Schl\"{a}fli symbol if $m \geq 3$.
\end{abstract}

\maketitle

\section{Introduction}
\label{sec:introduction}

Although all symmetries of a space form a group, its subset which consists of particular symmetries does not in general.
For instance, the reflective symmetries of a regular polygon do not form a group by themselves.
On the other hand, they form a quandle which is an algebraic system.
This quandle is called a dihedral quandle, if the polygon has an odd number of sides.
Various kinds of specified symmetries form quandles.
For example, the $(2 \pi / 3)$-rotations of a regular tetrahedron about axes passing through its center and the vertices form the tetrahedral quandle.
We also have hexahedral, octahedral, dodecahedral, and icosahedral quandles in the same manner.

Regular polyhedra are identified with regular tessellations of the 2-dimensional spherical space $\mathbb{S}^{2}$.
Two copies of a regular polygon also tessellate $\mathbb{S}^{2}$ regularly.
In this view, each of dihedral and polyhedral quandles is considered as a quandle consisting of the rotations of $\mathbb{S}^{2}$ about the vertices of a regular tessellation by the same angle which preserve the tessellation setwise.
In a similar way, we may have quandles related to regular tessellations of spherical, Euclidean, and hyperbolic spaces.
Since those tessellations are characterized by Schl\"{a}fli symbols, we call them Schl\"{a}fli quandles\footnote{The author has called a Schl\"{a}fli quandle a mosaic quandle in early version of this paper \cite{Inoue2018}.}.
In this paper, we focus on Schl\"{a}fli quandles related to regular tessellations $\{ 3, m \}$ ($m \geq 2$), $\{ 3, 3, 4 \}$, $\{ 3, 4, 3 \}$ and $\{ 3, 3, 5 \}$ in the sense of Schl\"{a}fli symbols.
Since the latter three tessellations are respectively identified with 16-, 24-, and 600-cells, we call correspondent Schl\"{a}fli quandles 16-, 24-, and 600-cells quandles respectively.

Quandles also have a good chemistry with knot theory.
Associated with a knot, we have its knot quandle in a similar way to its knot group.
In this paper, we focus on the knot quandle of the $m$-twist-spun trefoil.
Here, the $m$-twist-spun trefoil is a typical 2-knot, i.e., a 2-sphere embedded in the standard 4-sphere smoothly and locally flatly.
We see that the 16-, 24-, or 600-cell quandle is respectively isomorphic to the knot quandle of the 3-, 4-, or 5-twist-spun trefoil (Theorem \ref{thm:main1}).
It is known by Clark et al.\ \cite{CSV2016} with computer calculation that the knot quandle of the 3- or 4-twist-spun trefoil is a central extension of the Schl\"{a}fli quandle related to $\{ 3, 3 \}$ or $\{ 3, 4 \}$ respectively, in terms of this paper.
We show that this relationship between the knot quandle of the $m$-twist-spun trefoil and the Schl\"{a}fli quandle related to $\{ 3, m \}$ is lasting forever, i.e., the knot quandle of the $m$-twist-spun trefoil is a central extension of the Schl\"{a}fli quandle related to $\{ 3, m \}$ if $m \geq 3$ (Theorem \ref{thm:main2}).

\section{Quandle}
\label{sec:quandle}

In this section, we recall some notions about quandles briefly.
We refer the reader to \cite{Kamada2017} for more details.

A \emph{quandle} is a non-empty set $X$ equipped with a binary operation $\ast : X \times X \rightarrow X$ satisfying the following axioms:
\begin{itemize}
\item[(Q1)]
For each $x \in X$, $x \ast x = x$
\item[(Q2)]
For each $x \in X$, the map $\ast \, x : X \rightarrow X$ ($w \mapsto w \ast x$) is bijective
\item[(Q3)]
For each $x, y, z \in X$, $(x \ast y) \ast z = (x \ast z) \ast (y \ast z)$
\end{itemize}

Let us see a few examples of quandles.
Consider a set $X$ which consists of the vertices of a regular $n$-gon $P$ in $\mathbb{R}^{2}$ ($n \geq 3$).
For each $v \in X$, let $l_{v}$ be the line passing through the center of $P$ and $v$.
Then for each $v, w \in X$ we have $v \ast w \in X$ which is the image of $v$ by the reflection through $l_{w}$.
It is easy to check that $\ast$ satisfies the axioms of a quandle.
We call this quandle the \emph{dihedral quandle} of order $n$.

Similarly, let $X$ be a set consisting of the vertices of a regular polyhedron $P$ in $\mathbb{R}^{3}$ and $l_{v}$ the line passing through the center of $P$ and $v \in X$.
Suppose that $\theta$ is $\pi / 2$ if $P$ is an octahedron, $2 \pi / 5$ if $P$ is an icosahedron, otherwise $2 \pi / 3$.
Then for each $v, w \in X$ we have $v \ast w \in X$ which is the image of $v$ by the $\theta$-rotation about $l_{w}$ (counterclockwise when we see the center from $w$).
It is routine to see that $\ast$ satisfies the axioms of a quandle.
We call this quandle the \emph{tetrahedral}, \emph{hexahedral}, \emph{octahedral}, \emph{dodecahedral}, or \emph{icosahedral quandle} respectively, if $P$ is a tetrahedron, hexahedron, octahedron, dodecahedron, or icosahedron.

The notions of homomorphism, epimorphism, isomorphism and automorphism are appropriately defined for quandles.
Suppose that $X$ is a quandle.
Axioms (Q2) and (Q3) assert that for each $x \in X$ the map $\ast \, x$ is an automorphism of $X$.
We call this automorphism an \emph{inner automorphism} of $X$.
The inner automorphisms of $X$ generate the \emph{inner automorphism group} $\mathrm{Inn(X)}$ of $X$ which is a subgroup of the automorphism group of $X$.
We call $X$ to be \emph{connected} if $\mathrm{Inn(X)}$ acts transitively on $X$.

\section{Schl\"{a}fli quandle}
\label{sec:Schlafli_quandle}

In this section, we introduce the concept of Schl\"{a}fli quandles.
Then we concretely construct Schl\"{a}fli quandles related to regular tessellations $\{ 3, m \}$ ($m \geq 2$), $\{ 3, 3, 4 \}$, $\{ 3, 4, 3 \}$ and $\{ 3, 3, 5 \}$ in the sense of Schl\"{a}fli symbols.

Consider a regular tessellation $T$ of a spherical, Euclidean, or hyperbolic space $\mathbb{B}$.
Let $V$ be the set consisting of vertices of $T$.
For each $v \in V$ choose a rotation $r_{v}$ of $\mathbb{B}$ which fixes $v$ and preserves $T$ setwise.
If we have
\begin{equation}
 r_{r_{w}(v)} = r_{w} \circ r_{v} \circ r_{w}^{-1} \label{eq:condition}
\end{equation}
for each $v, w \in V$, consider the set $\{ (v, r_{v}) \mid v \in V \}$ and define its binary operation $\ast$ by $(v, r_{v}) \ast (w, r_{w}) = (r_{w}(v), r_{r_{w}(v)})$.
It is easy to see that $\ast$ satisfies the axioms of a quandle.
We call this quandle a \emph{Schl\"{a}fli quandle} related to $T$.

Let us consider some concrete Schl\"{a}fli quandles.
We first focus on the regular tessellation $\{ 3, m \}$ ($m \geq 2$).
We note that $\{ 3, m \}$ tessellates $\mathbb{S}^{2}$ if $2 \leq m \leq 5$, the Euclidean plane if $m = 6$, otherwise the hyperbolic plane (see Figure \ref{fig:tesselations}).
For each vertex $v$ of $\{ 3, m \}$, let $r_{v}$ be the rotation about $v$ by the angle $2 \pi / m$.
Then the condition (\ref{eq:condition}) is obviously satisfied for each vertices $v, w$ of $\{ 3, m \}$.
We thus have a Schl\"{a}fli quandle related to $\{ 3, m \}$.
We note that the Schl\"{a}fli quandles related to $\{ 3, 2 \}$, $\{ 3, 3 \}$, $\{ 3, 4 \}$ and $\{ 3, 5 \}$ are obviously isomorphic to the dihedral quandle of order 3, tetrahedral, octahedral and icosahedral quandles respectively.
\begin{figure}[htbp]
 \centering
 \includegraphics[scale=0.25]{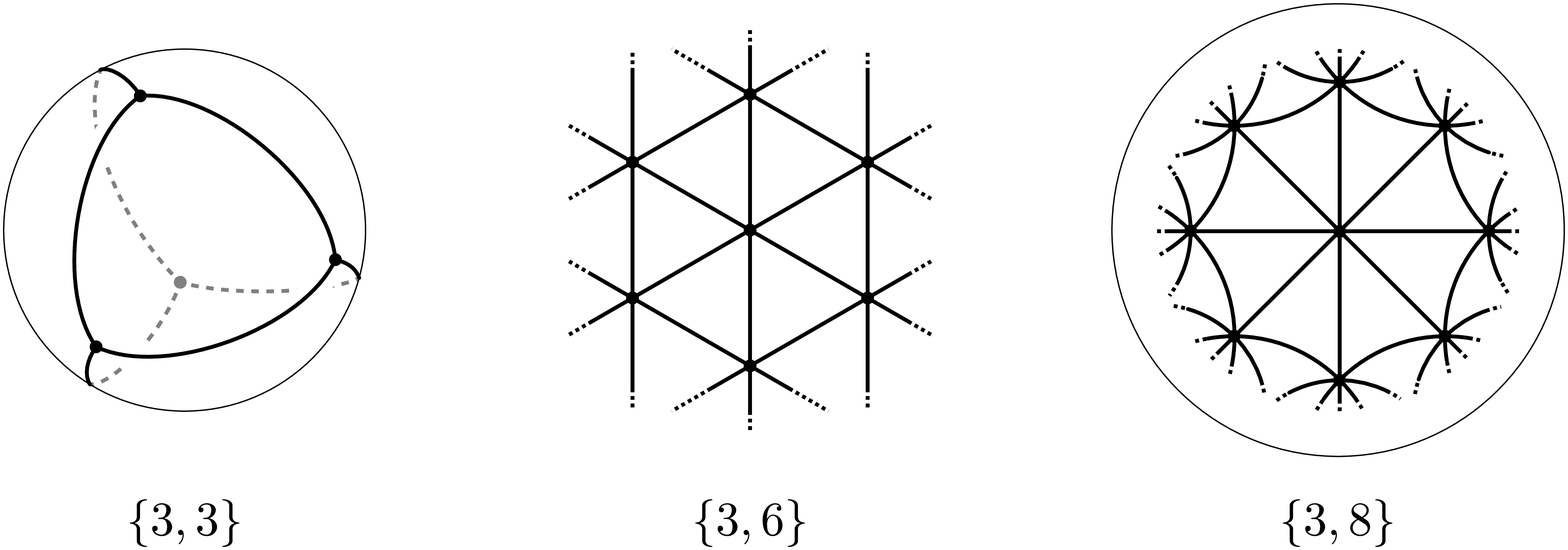}
 \caption{Regular tessellations $\{ 3, 3 \}$, $\{ 3, 6 \}$ and $\{ 3, 8 \}$, for example}
 \label{fig:tesselations}
\end{figure}

\begin{remark}
Regular tessellations $\{ 3, m \}$ converge to the Farey tessellation $\{ 3, \infty \}$ as $m$ goes to infinity.
We thus have a Schl\"{a}fli quandle related to $\{ 3, \infty \}$ as the ``limit'' of the Schl\"{a}fli quandles $\{ 3, m \}$.
Considering a famous relationship between $\{ 3, \infty \}$ and the mapping class group of a torus, it is routine to see that the Schl\"{a}fli quandle related to $\{ 3, \infty \}$ is isomorphic to the Dehn quandle of a torus (see \cite{NP2009}, for example, for a Dehn quandle).
It is known by Niebrzydowski and Przytycki \cite{NP2009} that the Dehn quandle of a torus (i.e., the Schl\"{a}fli quandle related to $\{ 3, \infty \}$) is isomorphic to the knot quandle of the trefoil.
In contrast with the fact, we will see that the knot quandle of the $m$-twist-spun trefoil is a central extension of the Schl\"{a}fli quandle related to $\{ 3, m \}$ if $m \geq 3$ (Theorem \ref{thm:main2}).
We note that the knot quandle of the 2-twist-spun trefoil is isomorphic to the Schl\"{a}fli quandle related to $\{ 3, 2 \}$ (Remark \ref{rem:knot_quandle_of_twist-spun_trefoil}).
\end{remark}

We next focus on regular tessellations $\{ 3, 3, 4 \}$, $\{ 3, 4, 3 \}$ and $\{ 3, 3, 5 \}$ of the \linebreak 3-dimensional spherical space $\mathbb{S}^{3}$.
In the remaining, we identify $\mathbb{S}^{3}$ with the unit sphere in $\mathbb{R}^{4}$.
We may assume that sets of the vertices of $\{ 3, 3, 4 \}$, $\{ 3, 4, 3 \}$ and $\{ 3, 3, 5 \}$ are respectively
\begin{align*}
 V_{\{ 3, 3, 4 \}}
 &= \{ \pm \vect{e}_{1}, \pm \vect{e}_{2}, \pm \vect{e}_{3}, \pm \vect{e}_{4} \}, \\
 V_{\{ 3, 4, 3 \}}
 &= \{ \pm \, \vect{e}_{i} \pm \vect{e}_{j} \mid 1 \leq i < j \leq 4 \}, \\
 V_{\{ 3, 3, 5 \}}
 &= \{ \pm \vect{e}_{1}, \pm \vect{e}_{2}, \pm \vect{e}_{3}, \pm \vect{e}_{4} \}
 \cup \left\{ \dfrac{1}{2} (\pm \, \vect{e}_{1} \pm \vect{e}_{2} \pm \vect{e}_{3} \pm \vect{e}_{4}) \right\} \\[1ex]
 & \qquad \cup \left\{ \left. \dfrac{1}{2} \left( \pm \, \phi \hskip 0.15em \vect{e}_{\sigma(1)} \pm \vect{e}_{\sigma(2)} \pm \phi^{-1} \hskip 0.05em \vect{e}_{\sigma(3)} \right) \; \right| \, \sigma \in A_{4} \right\}.
\end{align*}
Here, $\vect{e}_{i} \in \mathbb{R}^{4}$ denotes the column vector whose $j$-th entry is $\delta_{ij}$, $\phi$ the golden ratio $\left( 1 + \sqrt{5} \right) / 2$, and $A_{4}$ the alternating group on $\{ 1, 2, 3, 4 \}$.
Associated with $v \in V_{\mathrm{S}}$ ($S \in \{ \{ 3, 3, 4 \}, \, \{ 3, 4, 3 \}, \, \{ 3, 3, 5 \} \}$), we define the $4 \times 4$ matrix $R_{v}$ as follows:

\medskip

\noindent
$\blacktriangleright$ $S = \{ 3, 3, 4 \}$
\begin{align*}
 & R_{\pm \vect{e}_{1}}
 = \begin{pmatrix} 1 & 0 & 0 & 0 \\ 0 & 0 & 0 & 1 \\ 0 & 1 & 0 & 0 \\ 0 & 0 & 1 & 0 \end{pmatrix}, &
 & R_{\pm \vect{e}_{2}}
 = \begin{pmatrix} 0 & 0 & -1 & 0 \\ 0 & 1 & 0 & 0 \\ 0 & 0 & 0 & -1 \\ 1 & 0 & 0 & 0 \end{pmatrix}, \\[1ex]
 & R_{\pm \vect{e}_{3}}
 = \begin{pmatrix} 0 & 0 & 0 & -1 \\ 1 & 0 & 0 & 0 \\ 0 & 0 & 1 & 0 \\ 0 & -1 & 0 & 0 \end{pmatrix}, &
 & R_{\pm \vect{e}_{4}}
 = \begin{pmatrix} 0 & -1 & 0 & 0 \\ 0 & 0 & -1 & 0 \\ 1 & 0 & 0 & 0 \\ 0 & 0 & 0 & 1 \end{pmatrix}.
\end{align*}

\medskip

\noindent
$\blacktriangleright$ $S = \{ 3, 4, 3 \}$
\begin{align*}
 & R_{\pm(\vect{e}_{1} + \vect{e}_{2})}
 = R_{\pm(\vect{e}_{1} - \vect{e}_{2})}
 = \begin{pmatrix} 1 & 0 & 0 & 0 \\ 0 & 1 & 0 & 0 \\ 0 & 0 & 0 & -1 \\ 0 & 0 & 1 & 0 \end{pmatrix}, \\[1ex] 
 & R_{\pm(\vect{e}_{3} + \vect{e}_{4})}
 = R_{\pm(\vect{e}_{3} - \vect{e}_{4})}
 = \begin{pmatrix} 0 & -1 & 0 & 0 \\ 1 & 0 & 0 & 0 \\ 0 & 0 & 1 & 0 \\ 0 & 0 & 0 & 1 \end{pmatrix}, \\[1ex] 
 & R_{\pm(\vect{e}_{1} + \vect{e}_{3})}
 = R_{\pm(\vect{e}_{2} + \vect{e}_{4})}
 = \dfrac{1}{2} \begin{pmatrix} 1 & -1 & 1 & 1 \\ 1 & 1 & -1 & 1 \\ 1 & 1 & 1 & -1 \\ -1 & 1 & 1 & 1 \end{pmatrix}, \\[1ex] 
 & R_{\pm(\vect{e}_{1} - \vect{e}_{3})}
 = R_{\pm(\vect{e}_{2} - \vect{e}_{4})}
 = \dfrac{1}{2} \begin{pmatrix} 1 & -1 & -1 & -1 \\ 1 & 1 & 1 & -1 \\ -1 & -1 & 1 & -1 \\ 1 & -1 & 1 & 1 \end{pmatrix}, \\[1ex] 
 & R_{\pm(\vect{e}_{1} + \vect{e}_{4})}
 = R_{\pm(\vect{e}_{2} - \vect{e}_{3})}
 = \dfrac{1}{2} \begin{pmatrix} 1 & -1 & -1 & 1 \\ 1 & 1 & -1 & -1 \\ 1 & -1 & 1 & -1 \\ 1 & 1 & 1 & 1 \end{pmatrix}, \\[1ex] 
 & R_{\pm(\vect{e}_{2} + \vect{e}_{3})}
 = R_{\pm(\vect{e}_{1} - \vect{e}_{4})}
 = \dfrac{1}{2} \begin{pmatrix} 1 & -1 & 1 & -1 \\ 1 & 1 & 1 & 1 \\ -1 & 1 & 1 & -1 \\ -1 & -1 & 1 & 1 \end{pmatrix}. 
\end{align*}

\newpage

\noindent
$\blacktriangleright$ $S = \{ 3, 3, 5 \}$
\begin{align*}
 & R_{\pm \vect{e}_{1}}
 = R_{\pm \frac{1}{2} (\phi \hskip 0.07em \vect{e}_{1} + \vect{e}_{2} + \phi^{-1} \vect{e}_{3})}
 = R_{\pm \frac{1}{2} (\phi \hskip 0.07em \vect{e}_{1} - \vect{e}_{2} - \phi^{-1} \vect{e}_{3})} \\
 & \quad = R_{\pm \frac{1}{2} (\phi^{-1} \vect{e}_{1} + \phi \hskip 0.07em \vect{e}_{2} + \vect{e}_{3})}
 = R_{\pm \frac{1}{2} (\phi^{-1} \vect{e}_{1} - \phi \hskip 0.07em \vect{e}_{2} - \vect{e}_{3})}
 = \dfrac{1}{2} \begin{pmatrix} 2 & 0 & 0 & 0 \\ 0 & 1 & \phi & - \phi^{-1} \\ 0 & \phi & - \phi^{-1} & 1 \\ 0 & \phi^{-1} & -1 & - \phi \end{pmatrix}, \\[1ex] 
 & R_{\pm \vect{e}_{2}}
 = R_{\pm \frac{1}{2} (\vect{e}_{1} + \phi \hskip 0.07em \vect{e}_{2} + \phi^{-1} \vect{e}_{4})}
 = R_{\pm \frac{1}{2} (\vect{e}_{1} - \phi \hskip 0.07em \vect{e}_{2} + \phi^{-1} \vect{e}_{4})} \\
 & \quad = R_{\pm \frac{1}{2} (\phi \hskip 0.07em \vect{e}_{1} + \phi^{-1} \vect{e}_{2} + \vect{e}_{4})}
 = R_{\pm \frac{1}{2} (\phi \hskip 0.07em \vect{e}_{1} - \phi^{-1} \vect{e}_{2} + \vect{e}_{4})}
 = \dfrac{1}{2} \begin{pmatrix} 1 & 0 & \phi^{-1} & \phi \\ 0 & 2 & 0 & 0 \\ - \phi^{-1} & 0 & - \phi & 1 \\ \phi & 0 & -1 & - \phi^{-1} \end{pmatrix}, \\[1ex] 
 & R_{\pm \vect{e}_{3}}
 = R_{\pm \frac{1}{2} (\vect{e}_{1} + \phi^{-1} \vect{e}_{3} - \phi \hskip 0.07em \vect{e}_{4})}
 = R_{\pm \frac{1}{2} (\vect{e}_{1} - \phi^{-1} \vect{e}_{3} - \phi \hskip 0.07em \vect{e}_{4})} \\
 & \quad = R_{\pm \frac{1}{2} (\phi^{-1} \vect{e}_{1} + \phi \hskip 0.07em \vect{e}_{3} - \vect{e}_{4})}
 = R_{\pm \frac{1}{2} (\phi^{-1} \vect{e}_{1} - \phi \hskip 0.07em \vect{e}_{3} - \vect{e}_{4})}
 = \dfrac{1}{2} \begin{pmatrix} - \phi^{-1} & 1 & 0 & - \phi \\ -1 & - \phi & 0 & - \phi^{-1} \\ 0 & 0 & 2 & 0 \\ - \phi & \phi^{-1} & 0 & 1 \end{pmatrix}, \\[1ex] 
 & R_{\pm \vect{e}_{4}}
 = R_{\pm \frac{1}{2} (\vect{e}_{2} - \phi \hskip 0.07em \vect{e}_{3} + \phi^{-1} \vect{e}_{4})}
 = R_{\pm \frac{1}{2} (\vect{e}_{2} - \phi \hskip 0.07em \vect{e}_{3} - \phi^{-1} \vect{e}_{4})} \\
 & \quad = R_{\pm \frac{1}{2} (\phi^{-1} \vect{e}_{2} - \vect{e}_{3} + \phi \hskip 0.07em \vect{e}_{4})}
 = R_{\pm \frac{1}{2} (\phi^{-1} \vect{e}_{2} - \vect{e}_{3} - \phi \hskip 0.07em \vect{e}_{4})}
 = \dfrac{1}{2} \begin{pmatrix} - \phi & 1 & \phi^{-1} & 0 \\ -1 & - \phi^{-1} & - \phi & 0 \\ - \phi^{-1} & - \phi & 1 & 0 \\ 0 & 0 & 0 & 2 \end{pmatrix}, \\[1ex] 
 & R_{\pm \frac{1}{2} (\vect{e}_{1} + \vect{e}_{2} - \vect{e}_{3} + \vect{e}_{4})}
 = R_{\pm \frac{1}{2} (\vect{e}_{1} - \phi^{-1} \vect{e}_{3} + \phi \hskip 0.07em \vect{e}_{4})}
 = R_{\pm \frac{1}{2} (\phi \hskip 0.07em \vect{e}_{2} - \phi^{-1} \vect{e}_{3} - \vect{e}_{4})} \\
 & \quad = R_{\pm \frac{1}{2} (\phi^{-1} \vect{e}_{1} + \phi \hskip 0.07em \vect{e}_{2} - \vect{e}_{3})}
 = R_{\pm \frac{1}{2} (\phi^{-1} \vect{e}_{1} - \vect{e}_{2} + \phi \hskip 0.07em \vect{e}_{4})}
 = \dfrac{1}{2} \begin{pmatrix} - \phi^{-1} & 1 & 0 & \phi \\ 0 & 1 & - \phi & - \phi^{-1} \\ - \phi & -1 & - \phi^{-1} & 0 \\ 1 & -1 & -1 & 1\end{pmatrix}, \\[1ex] 
 & R_{\pm \frac{1}{2} (\vect{e}_{1} - \vect{e}_{2} - \vect{e}_{3} - \vect{e}_{4})}
 = R_{\pm \frac{1}{2} (\vect{e}_{1} + \phi^{-1} \vect{e}_{2} + \phi \hskip 0.07em \vect{e}_{3})}
 = R_{\pm \frac{1}{2} (\phi \hskip 0.07em \vect{e}_{1} - \phi^{-1} \vect{e}_{2} - \vect{e}_{4})} \\
 & \quad = R_{\pm \frac{1}{2} (\vect{e}_{2} + \phi \hskip 0.07em \vect{e}_{3} + \phi^{-1} \vect{e}_{4})}
 = R_{\pm \frac{1}{2} (\phi \hskip 0.07em \vect{e}_{1} + \vect{e}_{3} - \phi^{-1} \vect{e}_{4})}
 = \dfrac{1}{2} \begin{pmatrix} 1 & 0 & \phi^{-1} & - \phi \\ -1 & - \phi^{-1} & \phi & 0 \\ 1 & 1 & 1 & 1 \\ -1 & \phi & 0 & - \phi^{-1} \end{pmatrix}, \\[1ex] 
 & R_{\pm \frac{1}{2} (\vect{e}_{1} - \vect{e}_{2} + \vect{e}_{3} - \vect{e}_{4})}
 = R_{\pm \frac{1}{2} (\vect{e}_{1} - \phi^{-1} \vect{e}_{2} + \phi \hskip 0.07em \vect{e}_{3})}
 = R_{\pm \frac{1}{2} (\phi^{-1} \vect{e}_{2} + \vect{e}_{3} + \phi \hskip 0.07em \vect{e}_{4})} \\
 & \quad = R_{\pm \frac{1}{2} (\phi^{-1} \vect{e}_{1} + \phi \hskip 0.07em \vect{e}_{3} + \vect{e}_{4})}
 = R_{\pm \frac{1}{2} (\phi^{-1} \vect{e}_{1} - \vect{e}_{2} - \phi \hskip 0.07em \vect{e}_{4})}
 = \dfrac{1}{2} \begin{pmatrix} - \phi^{-1} & 0 & \phi & -1 \\ - \phi & - \phi^{-1} & 0 & 1 \\ 1 & -1 & 1 & 1 \\ 0 & \phi & \phi^{-1} & 1 \end{pmatrix}, \\[1ex] 
 & R_{\pm \frac{1}{2} (\vect{e}_{1} - \vect{e}_{2} - \vect{e}_{3} + \vect{e}_{4})}
 = R_{\pm \frac{1}{2} (\vect{e}_{1} + \phi \hskip 0.07em \vect{e}_{2} - \phi^{-1} \vect{e}_{4})}
 = R_{\pm \frac{1}{2} (\phi \hskip 0.07em \vect{e}_{1} - \vect{e}_{3} + \phi^{-1} \vect{e}_{4})} \\
 & \quad = R_{\pm \frac{1}{2} (\phi \hskip 0.07em \vect{e}_{1} + \vect{e}_{2} - \phi^{-1} \vect{e}_{3})}
 = R_{\pm \frac{1}{2} (\phi \hskip 0.07em \vect{e}_{2} + \phi^{-1} \vect{e}_{3} - \vect{e}_{4})}
 = \dfrac{1}{2} \begin{pmatrix} 1 & 1 & -1 & 1 \\ \phi^{-1} & 1 & 0 & - \phi \\ - \phi & 1 & - \phi^{-1} & 0 \\ 0 & -1 & - \phi & - \phi^{-1} \end{pmatrix}, \\[1ex] 
\end{align*}
\begin{align*}
 & R_{\pm \frac{1}{2} (\vect{e}_{1} + \vect{e}_{2} - \vect{e}_{3} - \vect{e}_{4})}
 = R_{\pm \frac{1}{2} (\vect{e}_{1} + \phi^{-1} \vect{e}_{2} - \phi \hskip 0.07em \vect{e}_{3})}
 = R_{\pm \frac{1}{2} (\phi^{-1} \vect{e}_{2} + \vect{e}_{3} - \phi \hskip 0.07em \vect{e}_{4})} \\
 & \quad = R_{\pm \frac{1}{2} (\phi^{-1} \vect{e}_{1} + \vect{e}_{2} - \phi \hskip 0.07em \vect{e}_{4})}
 = R_{\pm \frac{1}{2} (\phi^{-1} \vect{e}_{1} - \phi \hskip 0.07em \vect{e}_{3} + \vect{e}_{4})}
 = \dfrac{1}{2} \begin{pmatrix} - \phi^{-1} & \phi & -1 & 0 \\ 0 & - \phi^{-1} & -1 & - \phi \\ - \phi & 0 & 1 & - \phi^{-1} \\ -1 & -1 & -1 & 1 \end{pmatrix}, \\[1ex] 
 & R_{\pm \frac{1}{2} (\vect{e}_{1} + \vect{e}_{2} + \vect{e}_{3} + \vect{e}_{4})}
 = R_{\pm \frac{1}{2} (\vect{e}_{1} - \phi \hskip 0.07em \vect{e}_{2} - \phi^{-1} \vect{e}_{4})}
 = R_{\pm \frac{1}{2} (\phi \hskip 0.07em \vect{e}_{1} + \vect{e}_{3} + \phi^{-1} \vect{e}_{4})} \\
 & \quad = R_{\pm \frac{1}{2} (\phi \hskip 0.07em \vect{e}_{2} + \phi^{-1} \vect{e}_{3} + \vect{e}_{4})}
 = R_{\pm \frac{1}{2} (\phi \hskip 0.07em \vect{e}_{1} - \vect{e}_{2} + \phi^{-1} \vect{e}_{3})}
 = \dfrac{1}{2} \begin{pmatrix} 1 & - \phi^{-1} & \phi & 0 \\ -1 & 1 & 1 & 1 \\ 1 & 0 & - \phi^{-1} & \phi \\ 1 & \phi & 0 & - \phi^{-1} \end{pmatrix}, \\[1ex] 
 & R_{\pm \frac{1}{2} (\vect{e}_{1} - \vect{e}_{2} + \vect{e}_{3} + \vect{e}_{4})}
 = R_{\pm \frac{1}{2} (\vect{e}_{1} + \phi^{-1} \vect{e}_{3} + \phi \hskip 0.07em \vect{e}_{4})}
 = R_{\pm \frac{1}{2} (\phi \hskip 0.07em \vect{e}_{2} - \phi^{-1} \vect{e}_{3} + \vect{e}_{4})} \\
 & \quad = R_{\pm \frac{1}{2} (\phi^{-1} \vect{e}_{1} + \vect{e}_{2} + \phi \hskip 0.07em \vect{e}_{4})}
 = R_{\pm \frac{1}{2} (\phi^{-1} \vect{e}_{1} - \phi \hskip 0.07em \vect{e}_{2} + \vect{e}_{3})}
 = \dfrac{1}{2} \begin{pmatrix} - \phi^{-1} & 0 & \phi & 1 \\ -1 & 1 & -1 & 1 \\ 0 & - \phi & - \phi^{-1} & 1 \\ \phi & \phi^{-1} & 0 & 1 \end{pmatrix}, \\[1ex] 
 & R_{\pm \frac{1}{2} (\vect{e}_{1} + \vect{e}_{2} + \vect{e}_{3} - \vect{e}_{4})}
 = R_{\pm \frac{1}{2} (\vect{e}_{1} - \phi^{-1} \vect{e}_{2} - \phi \hskip 0.07em \vect{e}_{3})}
 = R_{\pm \frac{1}{2} (\phi \hskip 0.07em \vect{e}_{1} + \phi^{-1} \vect{e}_{2} - \vect{e}_{4})} \\
 & \quad = R_{\pm \frac{1}{2} (\vect{e}_{2} + \phi \hskip 0.07em \vect{e}_{3} - \phi^{-1} \vect{e}_{4})}
 = R_{\pm \frac{1}{2} (\phi \hskip 0.07em \vect{e}_{1} - \vect{e}_{3} - \phi^{-1} \vect{e}_{4})}
 = \dfrac{1}{2} \begin{pmatrix} 1 & 1 & -1 & -1 \\ 0 & - \phi^{-1} & 1 & - \phi \\ - \phi^{-1} & \phi & 1 & 0 \\ - \phi & 0 & -1 & - \phi^{-1} \end{pmatrix}. 
\end{align*}

\medskip

\noindent
For each $v \in V_{S}$, let $r_{v} : \mathbb{R}^{4} \rightarrow \mathbb{R}^{4}$ be the linear transformation sending $\vect{x}$ to $R_{v} \vect{x}$.
We note that $r_{v}$ is respectively a $(2 \pi / 3)$-, $(\pi / 2)$-, or $(2 \pi / 5)$-rotation about a plane in which $v$ and the origin of $\mathbb{R}^{4}$ lie, if $S$ is $\{ 3, 3, 4 \}$, $\{ 3, 4, 3 \}$ or $\{ 3, 3, 5 \}$.
It is routine to check that the condition (\ref{eq:condition}) is satisfied for each $v, w \in V_{S}$.
We thus have Schl\"{a}fli quandles related to $\{ 3, 3, 4 \}$, $\{ 3, 4, 3 \}$ and $\{ 3, 3, 5 \}$.
We call them \emph{16-}, \emph{24-}, and \emph{600-cells quandles} respectively, since convex hulls of $V_{\{ 3, 3, 4 \}}$, $V_{\{ 3, 4, 3 \}}$ and $V_{\{ 3, 3, 5 \}}$ in $\mathbb{R}^{4}$ are respectively known as 16-, 24-, and 600-cells.

\begin{remark}
\label{rem:RIG}
16- and 24-cells quandles are referred as $\SmallQuandle{8}{1}$ and $\SmallQuandle{24}{2}$ in {\sf GAP} package Rig \cite{Rig} respectively.
We note that they had no geometrical explanations before this.
\end{remark}

\begin{remark}
Reflections of a metric space also form a quandle under suitable conditions.
This quandle is called a Coxeter quandle (see \cite{FR1992} for example).
\end{remark}

\section{Presentation of a quandle}
\label{sec:presentation_of_a_quandle}

As well as groups, we have presentations of quandles.
Since we will utilize them for our arguments, we briefly recall some notions about presentations of quandles in this section.
We refer the reader to \cite{Kamada2017} for more details.

Let $S$ be a non-empty set, $F(S)$ the free group on $S$, and $FQ(S)$ the union of the conjugacy classes of $F(S)$ each of which contains an element of $S$.
It is easy to see that the binary operation $\ast$ on $FQ(S)$ given by
\[
 (g^{-1} s g) \ast (h^{-1} t h) = (g h^{-1} t h)^{-1} s (g h^{-1} t h) \quad (s, t \in S, \, g, h \in F(S))
\]
satisfies the axioms of a quandle.
We call this quandle the \emph{free quandle} on $S$.

For a given subset $R$ of $FQ(S) \times FQ(S)$, we consider to enlarge $R$ by repeating the following moves:
\begin{itemize}
\item[(a)]
For each $x \in FQ(S)$, add $(x, x)$ in $R$
\item[(b)]
For each $(x, y) \in R$, add $(y, x)$ in $R$
\item[(c)]
For each $(x, y), (y, z) \in R$, add $(x, z)$ in $R$
\item[(d)]
For each $(x, y) \in R$ and $s \in S$, add $(x \ast s, y \ast s)$ and $(x \; \overline{\ast} \; s, y \; \overline{\ast} \; s)$ in $R$
\item[(e)]
For each $(x, y) \in R$ and $z \in FQ(S)$, add $(z \ast x, z \ast y)$ and $(z \; \overline{\ast} \; x, z \; \overline{\ast} \; y)$ in $R$
\end{itemize}
Here, $x \; \overline{\ast} \; y$ denotes the element $(\ast \, y)^{-1}(x)$.
A \emph{consequence} of $R$ is an element of an expanded $R$ by a finite sequence of the above moves.
Let $x \sim_{R} y$ denote that $(x, y)$ is a consequence of $R$.
Then $\sim_{R}$ is obviously an equivalence relation on $FQ(S)$.
The quotient $FQ(S) / {\sim_{R}}$ inherits $\ast$ from $FQ(S)$, and $\ast$ on $FQ(S) / {\sim_{R}}$ still satisfies the axioms of a quandle.

A quandle $X$ is said to have a \emph{presentation} $\langle S \hskip 0.2em | \hskip 0.1em R \rangle$ if $X$ is isomorphic to the quandle $FQ(S) / {\sim_{R}}$.
We refer to an element of $S$ and $R$ as a \emph{generator} and a \emph{relation} of $\langle S \hskip 0.2em | \hskip 0.1em R \rangle$ respectively.
We write a consequence $x \sim_{R} y$ as $x = y$ and abbreviate a presentation $\langle \{ s_{1}, s_{2}, \dots, s_{n} \} \mid \{ r_{1}, r_{2}, \dots, r_{m} \} \rangle$ as $\langle s_{1}, s_{2}, \dots, s_{n} \mid r_{1}, r_{2}, \dots, r_{m} \rangle$ in the remaining.
Fenn and Rourke \cite{FR1992} essentially showed the following theorem which is similar to the Tietze's theorem for group presentations.
We will refer to the moves (T1), (T2) and their inverses as \emph{Tietze moves}.

\begin{theorem}
Assume that a quandle has two distinct presentations.
Then the presentations are related to each other by a finite sequence of the following moves or their inverses:
\begin{itemize}
\item[(T1)]
Choose a consequence of the set of relations, and then add it to the set of relations
\item[(T2)]
Choose an element $x$ of the free quandle on the set of generators, and then introduce a new generator $s$ in the set of generators and the new relation $s = x$ in the set of relations
\end{itemize}
\end{theorem}

Suppose that $S$ is a set and $X$ a quandle.
A map $f : S \rightarrow X$ naturally induces a homomorphism $f_{\sharp} : FQ(S) \rightarrow X$.
Further $f$ induces an well-defined homomorphism $f_{\ast} : FQ(S) / {\sim_{R}} \rightarrow X$ if we have $f_{\sharp}(x) = f_{\sharp}(y)$ for each relation $x = y$ in $R$.
We note that both $f_{\sharp}$ and $f_{\ast}$ are surjective if the image of $f$ generates $X$.

\section{Knot quandle of the twist-spun trefoil}
\label{sec:knot_quandle_of_the_twist-spun trefoil}

In this section, we review the twist-spun trefoil and its knot quandle rapidly.
We refer the reader \cite{Kamada2017} for more details.

Consider the oriented knotted arc $k$, depicted in the left-hand side of Figure \ref{fig:long_trefoil}, which is properly embedded in the upper half space $\mathbb{R}^{3}_{+}$.
Choose a 3-ball $B$ in $\mathbb{R}^{3}_{+}$ so that $B$ wholly contains the knotted part of $k$ (see the left-hand side of Figure \ref{fig:long_trefoil}).
We assume that $k$ intersects with $\partial B$ only at the north and south poles of $B$.
Suppose that $m$ is a positive integer.
Spin $\mathbb{R}^{3}_{+}$ 360 degrees in $\mathbb{R}^{4}$ along $\partial \mathbb{R}^{3}_{+}$, and simultaneously rotate $B$ $360 m$ degrees along the axis of $B$ passing through the north and south poles.
Then the locus of $k$ yields an oriented 2-knot after the one point compactification of $\mathbb{R}^{4}$.
We call this 2-knot the \emph{$m$-twist-spun trefoil}.
In the remaining, we let $\tau^{m} 3_{1}$ denote the $m$-twist-spun trefoil.
\begin{figure}[htbp]
 \centering
 \includegraphics[scale=0.25]{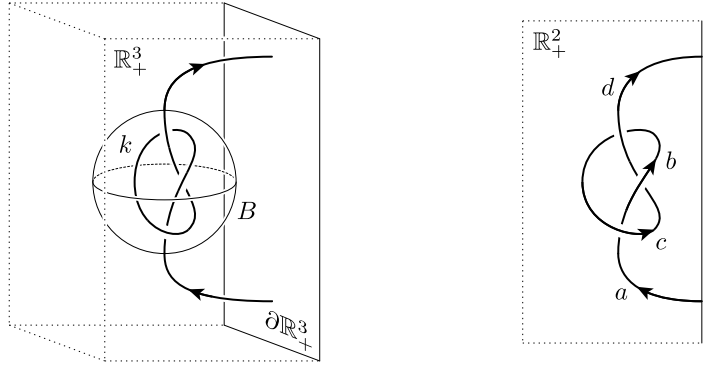}
 \caption{An oriented knotted arc $k$ called the long trefoil and a 3-ball $B$ which wholly contains the knotted part of $k$ (left), and a diagram of $k$ (right)}
 \label{fig:long_trefoil}
\end{figure}

It is known by Zeeman \cite{Zeeman1965} that $\tau^{m} 3_{1}$ is a fibered 2-knot (see \cite{Rolfsen1976}, for example, for fiberedness of a knot).
Therefore, in light of Corollary 3.2 of \cite{Inoue2019}, the knot quandle of $\tau^{m} 3_{1}$ is defined as follows, although it is different from the usual way.
Let $G_{m}$ be the fundamental group of a fiber of $\tau^{m} 3_{1}$, and $\varphi$ the monodromy of $\tau^{m} 3_{1}$ (which is an automorphism of $G_{m}$).
Then the \emph{knot quandle of the $m$-twist-spun trefoil} is defined to be $G_{m}$ equipped with the binary operation $\ast$ given by $x \ast y = \varphi(x y^{-1}) y$.
In the remaining, we let $Q_{m}$ denote the knot quandle of $\tau^{m} 3_{1}$, although $Q_{m}$ coincides with $G_{m}$ as sets.
We will count $G_{m}$ as a group and $Q_{m}$ as a quandle.

We here study about $G_{m}$ and $\varphi$ for subsequent arguments.
Since a fiber of $\tau^{m} 3_{1}$ has the surgery description depicted in Figure \ref{fig:surgery_description}, we have the presentation
\[
 \langle \gamma_{1}, \gamma_{2}, \dots, \gamma_{m} \mid \gamma_{1} = \gamma_{2} \gamma_{m}, \, \gamma_{2} = \gamma_{3} \gamma_{1}, \, \dots, \, \gamma_{m-1} = \gamma_{m} \gamma_{m-2}, \, \gamma_{m} = \gamma_{1} \gamma_{m-1} \rangle
\]
of $G_{m}$ (see \cite{Rolfsen1976}, for example, for a surgery description of a 3-manifold).
Here, $\gamma_{i}$ is the loop depicted in Figure \ref{fig:surgery_description}.
It is easy to see that $\varphi$ maps $\gamma_{i}$ to $\gamma_{i+1}$ ($1 \leq i \leq m-1$) and $\gamma_{m}$ to $\gamma_{1}$.
Let $\delta$ be the loop depicted in Figure \ref{fig:surgery_description}.
Since
\[
 \delta
 = \gamma_{2} \gamma_{1}^{-1} \gamma_{2}^{-1} \gamma_{1}
 = \dots
 = \gamma_{m} \gamma_{m-1}^{-1} \gamma_{m}^{-1} \gamma_{m-1}
 = \gamma_{1} \gamma_{m}^{-1} \gamma_{1}^{-1} \gamma_{m},
\]
$\varphi(\delta)$ and $\delta$ are the same element in $G_{m}$.
We note that $\delta \neq 1$ if and only if $m \geq 3$.
\begin{figure}[htbp]
 \centering
 \includegraphics[scale=0.25]{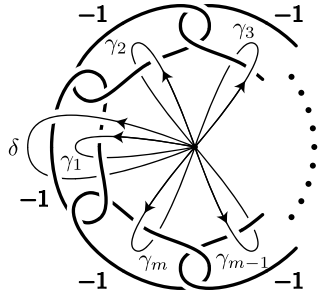}
 \caption{A surgery description of a fiber of $\tau^{m} 3_{1}$ (thick lines) and typical elements of $G_{m}$ (thin lines)}
 \label{fig:surgery_description}
\end{figure}

Since $k$ has the diagram depicted in the right-hand side of Figure \ref{fig:long_trefoil}, it is known by Satoh \cite{Satoh2002} that $Q_{m}$ has presentations
\begin{align}
 & \hskip 1.3em \langle a, b, c, d \mid a \ast c = b, \, b \ast d = c, \, c \ast b = d, \, b \ast^{m} a = b, \, c \ast^{m} a = c \rangle \notag \\
 & = \langle a, b, c \mid a \ast c = b, \, (b \ast c) \ast b = c, \, b \ast^{m} a = b, \, c \ast^{m} a = c \rangle \notag \\
 & = \langle a, c \mid (a \ast c) \ast a = c, \, c \ast^{m} a = c \rangle. \label{eq:presentation_of_knot_quandle}
\end{align}
Here, equality of presentations means being related to each other by Tietze moves, and $\ast^{m} \, a$ the $m$ times iteration of $\ast \, a$.
Since $Q_{m}$ is generated by the set $\{ a, c \}$ and $c$ is equal to $(a \ast c) \ast a$, $Q_{m}$ is obviously connected.
Studying works \cite{Inoue2019} and \cite{Satoh2002}, we know that
\[
 a = 1, \enskip
 b = \gamma_{m}, \enskip
 c = \gamma_{1}, \enskip
 d = \delta.
\]

\section{Relationships between Schl\"{a}fli quandles and knot quandles}
\label{sec:relationships_between_Schlafli_quandles_and_knot_quandles}

In this section, we study some relationships between Schl\"{a}fli quandles and knot quandles of twist-spun trefoils.
We start with showing the following theorem:

\begin{theorem}
\label{thm:main1}
The 16-, 24-, or 600-cell quandle is respectively isomorphic to the knot quandle of the 3-, 4-, or 5-twist-spun trefoil.
\end{theorem}

\begin{proof}
Choose adjacent vertices $v, w$ of the regular tessellation $\{ 3, 3, 4 \}$, $\{ 3, 4, 3 \}$ or $\{ 3, 3, 5 \}$ as follows:
\[
 (v, w) =
 \begin{cases}
  (\vect{e}_{1}, \vect{e}_{2}) & \text{if $\{ 3, 3, 4 \}$}, \\
  (\vect{e}_{1} + \vect{e}_{2}, \vect{e}_{2} + \vect{e}_{4}) & \text{if $\{ 3, 4, 3 \}$}, \\[0.5ex]
  \left( \vect{e}_{1}, - \dfrac{1}{2} \left( \phi^{-1} \hskip 0.05em \vect{e}_{1} + \phi \hskip 0.15em \vect{e}_{3} - \vect{e}_{4} \right) \right) & \text{if $\{ 3, 3, 5 \}$}.
 \end{cases}
\]
Let $X$ be the 16-, 24-, or 600-cell quandle, and $m$ respectively equal to 3, 4, or 5.
Then it is routine to check that the set $\{ (v, r_{v}), (w, r_{w}) \}$ generates $X$ and we have
\[
 ((v, r_{v}) \ast (w, r_{w})) \ast (v, r_{v}) = (w, r_{w}), \enskip
 (w, r_{w}) \ast^{m} (v, r_{v}) = (w, r_{w}).
\]
We thus have the epimorphism $f_{\ast} : Q_{m} \to X$ which sends $a$ and $c$ to $(v, r_{v})$ and $(w, r_{w})$ respectively.
It is known that the cardinality of $G_{m}$ (i.e., of $Q_{m}$) is equal to 8, 24, or 120 respectively (see Section 10.D of \cite{Rolfsen1976} for example).
Since this number is equal to the cardinality of $X$, $f_{\ast}$ is not only an epimorphism but an isomorphism.
\end{proof}

\begin{remark}
\label{rem:knot_quandle_of_twist-spun_trefoil}
Since $\tau^{1} 3_{1}$ is equivalent to the trivial 2-knot \cite{Zeeman1965}, $Q_{1}$ is the quandle of order 1.
Rourke and Sanderson \cite{RS2002} pointed out that $Q_{2}$ is isomorphic to the dihedral quandle of order 3 (i.e., the Schl\"{a}fli quandle related to $\{ 3, 2 \}$).
In light of Theorem 4.1 of \cite{Inoue2019}, we know that the cardinality of $Q_{m}$ is infinite if $m \geq 6$.
\end{remark}

We next make discussion on a central extension.
Let $\widetilde{X}$ and $X$ be quandles and $A$ a non-trivial abelian group.
Suppose that $A$ acts on $\widetilde{X}$ from the left.
Then $\widetilde{X}$ is said to be a \emph{central extension} of $X$ if there is an epimorphism $p : \widetilde{X} \to X$ satisfying the following conditions \cite{Eisermann2003}:
\begin{itemize}
\item[(E0)]
For each $\widetilde{w}, \widetilde{x}, \widetilde{y} \in \widetilde{X}$, $p(\widetilde{x}) = p(\widetilde{y})$ implies $\widetilde{w} \ast \widetilde{x} = \widetilde{w} \ast \widetilde{y}$
\item[(E1)]
For each $\widetilde{x}, \widetilde{y} \in \widetilde{X}$ and $\alpha \in A$, $(\alpha \widetilde{x}) \ast \widetilde{y} = \alpha (\widetilde{x} \ast \widetilde{y})$ and $\widetilde{x} \ast (\alpha \widetilde{y}) = \widetilde{x} \ast \widetilde{y}$
\item[(E2)]
For each $x \in X$, $A$ acts on the fiber $p^{-1}(x)$ freely and transitively
\end{itemize}
A central extension is also called an \emph{abelian extension} (see \cite{CENS2003} for example).
As well as groups, central extensions of a quandle are closely related to the second cohomology group of the quandle \cite{CENS2003, Eisermann2003}.

Recall that tetrahedral and octahedral quandles are the Schl\"{a}fli quandles related to $\{ 3, 3 \}$ and $\{ 3, 4 \}$ respectively.
It is known by Clark et al.\ \cite{CSV2016} with computer calculation that $\SmallQuandle{8}{1}$ and $\SmallQuandle{24}{2}$ are respectively central extensions of tetrahedral and octahedral quandles.
Thus, in light of Remark \ref{rem:RIG} and Theorem \ref{thm:main1}, we know that $Q_{m}$ is a central extension of the Schl\"{a}fli quandle related to $\{ 3, m \}$ if $m$ is equal to 3 or 4.
This relationship between $Q_{m}$ and the Schl\"{a}fli quandle related to $\{ 3, m \}$ is lasting as follows:

\begin{theorem}
\label{thm:main2}
The knot quandle of the $m$-twist-spun trefoil is a central extension of the Schl\"{a}fli quandle related to $\{ 3, m \}$ if $m \geq 3$.
\end{theorem}

To show the theorem, we first prepare the following lemma:

\begin{lemma}
\label{lem:presentation}
The Schl\"{a}fli quandle related to $\{ 3, m \}$ {\upshape (}$m \geq 2${\upshape )} has presentations
\begin{align}
 & \hskip 1.3em \langle v, w \mid (v \ast w) \ast v = w, \, (w \ast v) \ast w = v, \, w \ast^{m} v = w \rangle \notag \\
 & = \langle v, w \mid (v \ast w) \ast v = w, \, (((v \ast w) \ast v) \ast v) \ast w = v, \, w \ast^{m} v = w \rangle. \label{eq:presentation_of_Schlafli_quandle}
\end{align}
\end{lemma}

\begin{proof}
Consider the quandle $X$ having the former presentation.
Since we have the first relation and
\[
 w \ast (v \ast w)
 = (w \ast v) \ast w
 = v,
\]
the triple $(v, v \ast w, w) \in X^{3}$ forms a triangle as depicted in Figure \ref{fig:triangles} (a).
Further since we have the second relation and
\[
 v \ast (w \ast v)
 = (v \ast w) \ast v
 = w,
\]
the triple $(v, w, w \ast v) \in X^{3}$ also forms a triangle (see Figure \ref{fig:triangles} (b)).
In a similar way, each triple $(v, w \ast^{k-1} v, w \ast^{k} v) \in X^{3}$ ($2 \leq k \leq m$) again forms a triangle (see Figure \ref{fig:triangles} (c)).
Since we have the third relation, these triangles should be allocated as depicted in Figure \ref{fig:triangles} (d).
Recall that $\ast \, v$ is bijective.
\begin{figure}[htbp]
 \centering
 \includegraphics[scale=0.20]{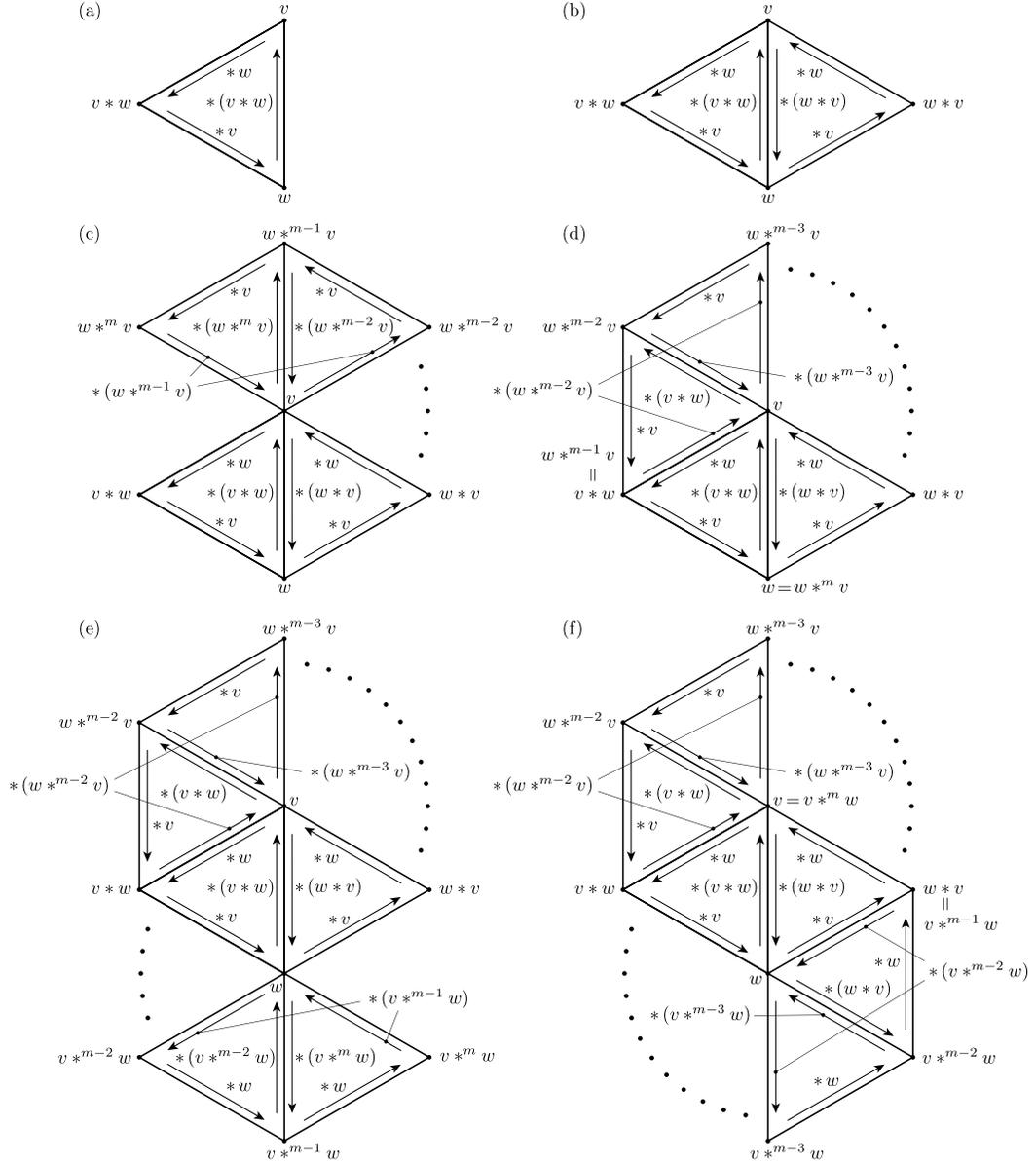}
 \caption{$X$ forms the regular tessellation $\{ 3, m \}$}
 \label{fig:triangles}
\end{figure}

In the same manner, each triple $(w, v \ast^{k-1} w, v \ast^{k} w) \in X^{3}$ ($1 \leq k \leq m$) forms a triangle (see Figure \ref{fig:triangles} (e)).
Since
\begin{align*}
 w \ast v
 & = (w \ast v) \ast (w \ast v) \\
 & = (w \ast^{m+1} v) \ast (w \ast v) \\
 & = (w \ast v) \ast^{m} (v \ast (w \ast v)) \\
 & = (w \ast v) \ast^{m} ((v \ast w) \ast v) \\
 & = (w \ast v) \ast^{m} w,
\end{align*}
the triangles should be allocated as depicted in Figure \ref{fig:triangles} (f).
Repeating a parallel argument sequentially, we know that there are just $m$ triangles around each element of $X$.
Since there is no such an arrangement of triangles other than $\{ 3, m \}$ up to ambient isotopy, $X$ should be isomorphic to the Schl\"{a}fli quandle related to $\{ 3, m \}$.

The latter presentation is obviously related to the former by Tietze moves.
\end{proof}

\begin{proof}[Proof of Theorem \ref{thm:main2}]
Let $X$ be the Schl\"{a}fli quandle related to $\{ 3, m \}$ ($m \geq 2$).
Since $Q_{m}$ and $X$ respectively have the presentations (\ref{eq:presentation_of_knot_quandle}) and (\ref{eq:presentation_of_Schlafli_quandle}), we have the epimorphism $p : Q_{m} \to X$ sending $a$ and $c$ to $v$ and $w$ respectively.

For each $g \in \mathrm{Inn}(Q_{m})$, let us consider the inner automorphism
\begin{align*}
 \widehat{g}
 & = (\ast \, g(c)) \circ (\ast \, g(a)) \circ (\ast \, g(a)) \circ (\ast \, g(c)) \\
 & = (\ast \, g(\gamma_{1})) \circ (\ast \, g(1)) \circ (\ast \, g(1)) \circ (\ast \, g(\gamma_{1}))
\end{align*}
of $Q_{m}$.
Then it is routine to check that we have
\begin{equation}
 \widehat{g}^{k}(g(1)) = g(\delta^{k}), \enskip
 \widetilde{w} \ast g(1) = \widetilde{w} \ast g(\delta^{k}) \label{eq:relations}
\end{equation}
for each $k \in \mathbb{Z}$ and $\widetilde{w} \in Q_{m}$.
Thus for each $x \in X$ and $\widetilde{x}, \widetilde{x}^{\prime} \in p^{-1}(x)$ there is an integer $k$ satisfying $\widetilde{x}^{\prime} = \delta^{k} \widetilde{x}$.
Indeed, we obtain a presentation of $X$ from (\ref{eq:presentation_of_knot_quandle}) by adding the relation $(((a \ast c) \ast a) \ast a) \ast c = a$.
The expressions (\ref{eq:relations}) tell us that this addition of the relation only introduces $g(\delta^{k}) = g(\delta^{l})$ for each $g \in \mathrm{Inn}(Q_{m})$ and $k, l \in \mathbb{Z}$ ($k \neq l$) as a new consequence of the set of relations.
It is easy to check that $g(\delta^{l})$ and $\delta^{l - k} g(\delta^{k})$ are the same element in $Q_{m}$, since $\varphi(\delta) = \delta$.
Recall that $Q_{m}$ is connected.

Let $A = \langle \delta \rangle$ be the abelian subgroup of $G_{m}$.
Then $A$ surely acts on $G_{m}$ (i.e., on $Q_{m}$) from the left.
Obviously $p$ satisfies conditions (E0) and (E2).
Further it is routine to see that $p$ also satisfies the condition (E1).
Since $A$ is non-trivial if $m \geq 3$, we obtain the claim.
\end{proof}

We conclude the paper with a question.
Suppose that $A$ is the above one.
Since both of cardinalities of $Q_{m}$ and the Schl\"{a}fli quandle related to $\{ 3, m \}$ are finite if $m$ is equal to 3, 4, or 5, we know that the order of $A$ is respectively 2, 4, or 10.
On the other hand, both of cardinalities of $Q_{m}$ and the Schl\"{a}fli quandle related to $\{ 3, m \}$ are infinite if $m \geq 6$.
What is the order of $A$ if $m \geq 6$?

\section*{Acknowledgments}
The author is supported by JSPS KAKENHI Grant Numbers JP16K17591 and JP19K03476 partially.

\bibliographystyle{amsplain}

\end{document}